\documentclass[%
 aip,
 amsmath,amssymb,
 reprint,%
]{revtex4-1}

\usepackage{graphicx}
\usepackage{dcolumn}
\usepackage{bm}
\usepackage{amsmath,amsthm} 
\usepackage{overpic}
\usepackage{tikz} 

\usepackage[utf8]{inputenc}
\usepackage[T1]{fontenc}
\usepackage{mathptmx}
\usepackage{etoolbox}

\usepackage{xcolor}

\newcommand{\at}[1]{{\color{blue}{AT: #1}}} 
\definecolor{red}{rgb}{.5,0,0} 
\definecolor{green}{rgb}{0,.5,0} 
\definecolor{blue}{rgb}{0,0,0.5}

\makeatletter
\def\@email#1#2{%
 \endgroup
 \patchcmd{\titleblock@produce}
  {\frontmatter@RRAPformat}
  {\frontmatter@RRAPformat{\produce@RRAP{*#1\href{mailto:#2}{#2}}}\frontmatter@RRAPformat}
  {}{}
}%
\makeatother

\newtheorem{thm}{Theorem}

\newtheorem{lemma}[thm]{Lemma}

\newtheorem{cor}[thm]{Corollary}

\newcommand{\bmu}{\widetilde{\mu}}

\begin{document}

\preprint{AIP/123-QED}

\title[Synchronization of dense networks]{Sufficiently dense Kuramoto networks are globally synchronizing}

\author{Martin Kassabov}
\altaffiliation{Department of Mathematics, Cornell University, Ithaca, NY 14853}%
\author{Steven H. Strogatz}
\altaffiliation{Department of Mathematics, Cornell University, Ithaca, NY 14853}%
\author{Alex Townsend}
\altaffiliation{Department of Mathematics, Cornell University, Ithaca, NY 14853}%

\date{\today}

\begin{abstract}
Consider any network of $n$ identical Kuramoto oscillators in which each oscillator is coupled bidirectionally with unit strength to at least $\mu (n-1)$ other oscillators. There is a critical value of the connectivity, $\mu_c$, such that whenever $\mu>\mu_c$, the system is guaranteed to converge to the all-in-phase synchronous state for almost all initial conditions, but when $\mu<\mu_c$, there are networks with other stable states. The precise value of the critical connectivity remains unknown, but it has been conjectured to be $\mu_c=0.75$. In 2020, Lu and Steinerberger proved that $\mu_c\leq 0.7889$, and Yoneda, Tatsukawa, and Teramae proved in 2021 that $\mu_c > 0.6838$.  In this paper, we prove that $\mu_c\leq 0.75$ and explain why this is the best upper bound that one can obtain by a purely linear stability analysis. 
\end{abstract}

\maketitle

\begin{quotation}
The Kuramoto model of coupled oscillators offers an ideal playground for exploring how the structure of a network affects the dynamics it can display. In that spirit, we consider the simplest version of the model, in which each oscillator is assumed to have the same intrinsic frequency. The oscillators are coupled with unit strength along the edges of an undirected network that is connected but otherwise has an arbitrary topology. Even in this minimalist setup, mysteries abound. Specifically, what level of connectivity ensures that the system will almost always settle into a state of perfect synchrony, with all the oscillators running in phase? Here we prove that if every oscillator is connected to at least $75\%$ of the others, the system globally synchronizes, regardless of the details of its wiring diagram. Our proof uses trigonometric identities and optimization arguments to derive inequalities (involving the first two moments of the oscillator phase distribution) that must hold for any stable phase-locked state. If it is possible to improve the bound $\mu_c\leq 0.75$ further, then one must contend with a remarkable sequence of networks that lie on the so-called razor's edge of stability. The networks in this sequence have connectivities that approach the upper bound of 0.75 from below, yet have twisted phase-locked states whose eigenvalues are either negative or zero; as such, linear analysis says nothing either way about their stability. Thus, although our theorem brings us closer to pinpointing the critical connectivity $\mu_c$, a quantity that has recently attracted intense theoretical interest, it does not settle the question. There still remains a large gap between the best known upper and lower bounds on $\mu_c$, and we are starting to suspect that this gap cannot be closed by linear stability analysis on its own. Indeed, if $\mu_c$ turns out to be less than 0.75, then nonlinear stability analysis will be needed to prove that this is the case. 
\end{quotation}

\section{\label{sec:level1}Introduction}
Coupled oscillators often synchronize spontaneously. This phenomenon appears throughout nature and technology, from flashing fireflies and neural populations to arrays of Josephson junctions and nanoelectromechanical oscillators.~\cite{winfree1967biological,peskin1975mathematical, kuramoto1984chemical, mirollo1990synchronization, watanabe1994constants, pikovsky2003synchronization, matheny2019exotic} 

Among the many questions raised by synchronization, one of the most natural asks how network topology can either promote or prevent global synchronization.~\cite{jadbabaie2004stability, wiley2006size,mallada2010synchronization, taylor2012there, dorfler2014synchronization,pecora2014cluster, pikovsky2015dynamics, canale2015exotic, Mehta15, rodrigues2016kuramoto,abrams2016introduction,deville2016phase, sokolov2018sync, ling2018landscape,lu2019synchronization, townsendstillmanstrogatz, Yoneda_2021} We say that a network of oscillators \emph{globally synchronizes} if it converges to a state for which all the oscillators are in phase, starting from all initial conditions except a set of measure zero. Otherwise, we say that the network \emph{supports a pattern}. One expects that dense networks are more inclined to be globally synchronizing, and sparser ones might support waves or more exotic patterns. But trusting intuition here can be dangerous. For example, a network in which any two oscillators are connected by exactly one path (i.e., a tree) can be quite sparse, yet all trees of identical Kuramoto oscillators are known to be globally synchronizing; conversely, there are dense Kuramoto networks that nevertheless support a pattern.~\cite{wiley2006size, canale2015exotic, townsendstillmanstrogatz, Yoneda_2021}

In this paper, we study the homogeneous Kuramoto model introduced by Taylor~\cite{taylor2012there}, in which each oscillator has the same frequency $\omega$. By going into a rotating frame at this frequency, we can set $\omega = 0$ without loss of generality. Then phase-locked states in the original frame correspond to equilibrium states in the rotating frame. So to explore the question that concerns us, it suffices to study the following simplified system of identical Kuramoto oscillators:
\begin{equation} 
\frac{d\theta_j}{dt} = \sum_{k=1}^{n} A_{jk} \sin\!\left(\theta_k - \theta_j\right), \quad 1\leq j\leq n. 
\label{eq:dynamical}
\end{equation} 
Here $\theta_j(t)$ is the phase of oscillator $j$ (in the rotating frame). The network's topology is encoded in the adjacency matrix $A$, with entries $A_{jk} = A_{kj} = 1$ if oscillator $j$ is coupled to oscillator $k$, and $A_{jk} = A_{kj} = 0$ otherwise. It is customary to assume that the network has no self-loops so that oscillator $i$ is not connected to itself, i.e., $A_{ii} = 0$. Thus, all interactions are assumed to be symmetric, equally attractive, and of unit strength.  The adjacency matrix $A$ is symmetric so~\eqref{eq:dynamical} is a gradient system.~\cite{wiley2006size,jadbabaie2004stability} Thus, all the attractors of~\eqref{eq:dynamical} are equilibrium points, which means we do not need to concern ourselves with the possibility of more complicated invariant sets like limit cycles, tori, or strange attractors. 

Taylor~\cite{taylor2012there} proved that the system~\eqref{eq:dynamical} globally synchronizes if each oscillator is coupled to at least 93.95\% of the others. This result started a flurry of research into finding the critical connectivity to guarantee global synchrony. To make this notion precise, we define the \emph{connectivity} $\mu$ of a network of size $n$ as the minimum degree of the nodes in the network, divided by $n-1$, the total number of other nodes. From here, we define the \emph{critical connectivity} $\mu_c$ as the smallest value of $\mu$ such that any network of $n$ oscillators is globally synchronizing. For any network $G$, if $\mu\geq\mu_c$ the network is guaranteed to be globally synchronizing. Otherwise, there exists a network with $\mu < \mu_c$ that supports a pattern. In these terms, Taylor's result is $\mu_c \le 0.9395$. 

The remarkable thing about the critical connectivity is that the wiring of the network can be arbitrary. Recently, Ling, Xu, and Bandeira~\cite{ling2018landscape} strengthened Taylor's result to show that $\mu_c\leq 0.7929$, and Lu and Steinerberger~\cite{lu2019synchronization} further refined the argument to show $\mu_c\leq 0.7889$. On the other hand, some very large dense networks with $\mu> 0.6838$ support a pattern.~\cite{Yoneda_2021} Thus, until the present paper, the best bounds were $0.6838< \mu_c \leq 0.7889$. In this paper, we improve the upper bound to $\mu_c\leq 0.75$.

Before turning to our proof that $\mu_c\leq 0.75$, we comment that any attempt to show $\mu_c<0.75$ with linear analysis (in the style of our argument) is most likely futile. In particular, the argument below cannot be refined to get a better upper bound on $\mu_c$. We know this because a sequence of networks exists with connectivity tending to $75\%$ that lie on the {\em razor's edge} of stability.~\cite{townsendstillmanstrogatz} Linear analysis alone cannot determine whether these networks are globally synchronizing or support a pattern; the difficulty is that the associated Jacobian of certain equilibrium states---the so-called \emph{twisted states}---has multiple zero eigenvalues.~\cite{townsendstillmanstrogatz} Unfortunately, long-time dynamical simulations reveal that these networks are most likely globally synchronizing, but just barely; they seem to avoid supporting a pattern by a whisker. Any potential argument that shows that $\mu_c<0.75$ must contend with this sequence of graphs. 

\section{Self-loops and twinning}\label{sec:twinning}
It is standard to assume that the nodes of the network in~\eqref{eq:dynamical} does not have self-loops. However, the absence of self-loops does cause a few peculiarities, which can be avoided if one assumes that each oscillator is connected to itself. The dynamical system in~\eqref{eq:dynamical} is oblivious to self-loops as when $j=k$ we have $\sin(\theta_k-\theta_j)=0$. Thus, it is only convention to take $A_{ii} = 0$ and the dynamics does not change when one adds self-loops, i.e., $A_{ii} = 1$.  However, self-loops do change how one measures the connectivity of a network as there are now $n$ possible neighbors. When we add self-loops to every node of a network, the connectivity jumps up from $\mu$ to 
\begin{equation} 
\bmu = \frac{\mu (n-1) + 1}{n}, 
\label{eq:newmu}
\end{equation} 
while the evolution of $\theta_1,\ldots,\theta_n$ over time is left unaltered.

There is a process known as twinning that is closely related and shows that the value of the critical connectivity is not affected by the presence or absence of self-loops. Given any graph $G$ (without self-loops) and the complete graph $K_\tau$ on $\tau$ nodes for any integer $\tau$, Canale and Monz{\'o}n~\cite{canale2015exotic} showed that the lexicographic product $G[K_\tau]$ has the same synchronizing behavior. That is, $G$ is globally synchronizing if and only if $G[K_\tau]$ is. And, $G$ supports a pattern if and only if $G[K_\tau]$ does. The graph $G[K_\tau]$ is formed by replacing each node in $G$ by a clique of size $\tau$, and the nodes in different cliques are connected just like the parent nodes they replaced. The twinned graph $G'=G[K_\tau]$ does not have self-loops and is denser than $G$. We find that $\mu' = (\tau-1+\tau \mu (n-1))/(n\tau-1)>\mu$ for any $\tau>1$, where $\mu'$ is the connectivity of $G'$. Thus, any graph $G$ with connectivity $\mu$ can be used as an initial seed to construct a sequence of denser graphs with a limiting edge density of $\left((n-1)\mu +1\right)/n$, where each graph in the sequence has the same synchronizing behavior as $G$ itself.  The limiting edge density of this sequence is $\bmu$: the same connectivity one can achieve by adding self-loops to all the nodes of $G$.

For this reason, the expressions we derive below simplify somewhat when we work with the parameter $\bmu$ in~\eqref{eq:newmu} instead of $\mu$. Our argument below shows that any network with $\bmu>0.75$ is guaranteed to be globally synchronizing (see Theorem~\ref{thm:MainResult}). Using~\eqref{eq:newmu}, this means that any network of size $n$ with connectivity $\mu > (0.75n -1)/(n-1)$ is globally synchronizing. Since $\mu (n-1)$ must be an integer (as it is the minimum degree of a node), we know that the following connectivity is sufficient for global synchrony:
\begin{equation} 
\mu > \frac{1}{n-1}\left\lfloor \frac{3n}{4} -1 \right\rfloor. 
\label{eq:TightestBound} 
\end{equation} 
In Figure~\ref{fig:Bound}, we plot the lower bound in~\eqref{eq:TightestBound} together with the connectivity of the densest known regular networks that support a pattern for $5\leq n\leq 50$.  The bound for $\mu$ in~\eqref{eq:TightestBound} tends to $0.75$, from below, and hence implies that $\mu_c\leq 0.75$.

For the rest of this paper, we consider a network of $n$ identical oscillators with self-loops that has connectivity $\bmu$. If we say that $\theta$ is a stable equilibrium, then we mean that $\theta$ is a stable state with respect to the dynamical system in~\eqref{eq:dynamical} for that network.

\begin{figure}
\centering
\begin{overpic}[width=.49\textwidth]{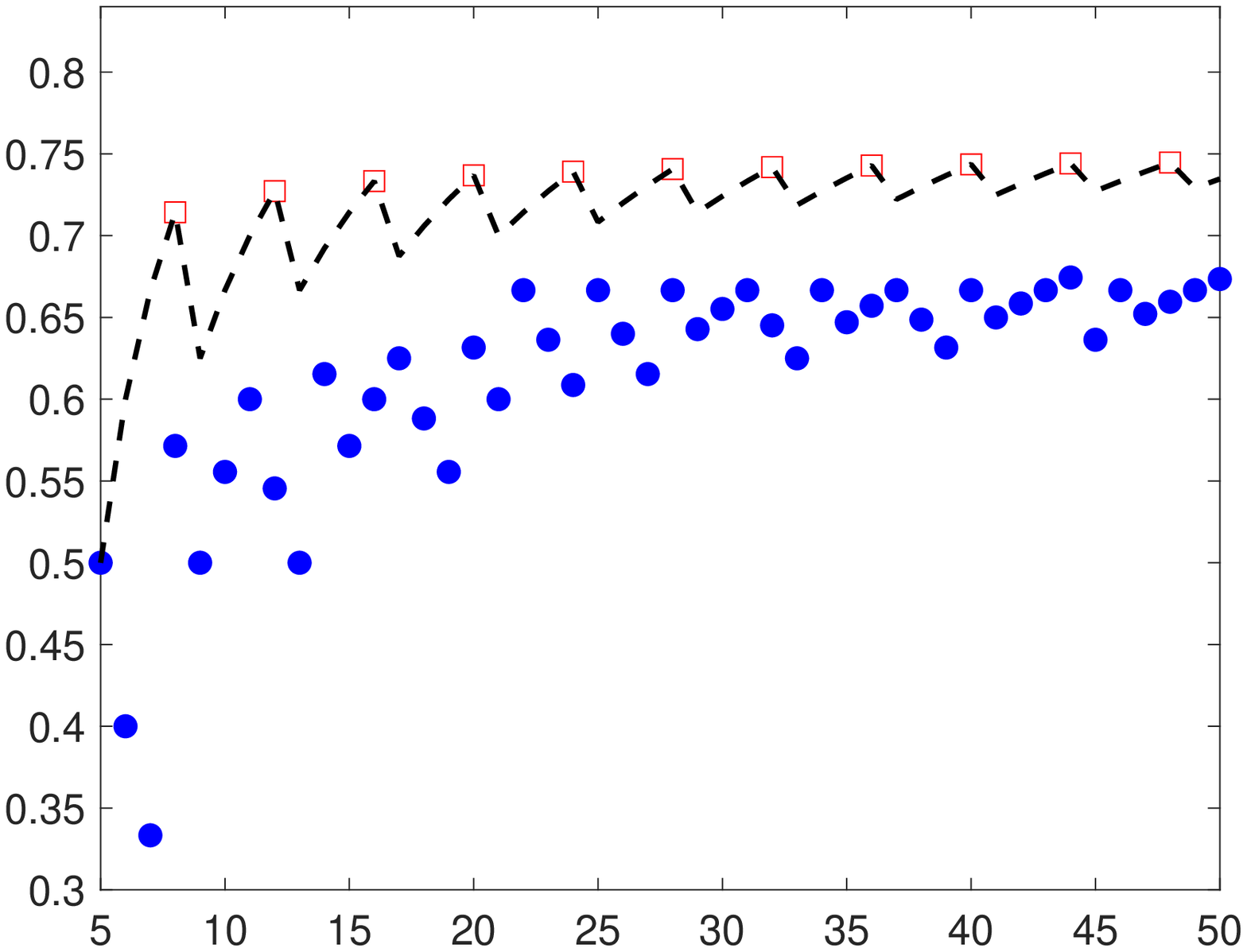}
\put(50,0) {$n$}
\put(2,35) {\rotatebox{90}{$\mu$}}
\end{overpic} 
\caption{Comparing the connectivity of the densest regular networks known to the authors that support a pattern~\cite{townsendstillmanstrogatz}  (blue dots) with the upper bound we prove (dashed line) for $5\leq n\leq 50$. Any network with connectivity above the dashed line is globally synchronizing, regardless of its wiring diagram. Note the tantalizing gap between the densest known networks that support a pattern and the upper bound, as also seen in the bounds $0.6838< \mu_c\leq 0.75$. When $n$ is a multiple of $4$, the dashed line cannot be lowered without contending with the networks represented by the red squares. These red squares correspond to a sequence of dense networks for which linear stability analysis alone cannot determine their global synchrony behavior.~\cite{townsendstillmanstrogatz}
}
\label{fig:Bound}
\end{figure}

\section{Order parameter and higher-order moments}
An important quantity in the study of Kuramoto oscillators is the so-called complex order parameter, $\rho_1$. Its magnitude $0 \leq |\rho_1| \leq 1$ measures the overall level of synchrony of the oscillators, and its argument measures their average phase.~\cite{kuramoto1984chemical, kuramoto} In geometrical terms, the complex order parameter corresponds to the centroid of the oscillators' positions on the unit circle, regarded as a subset of the complex plane. 

In the analysis below, we find it useful to look at higher-order moments $\rho_m$ of the oscillator distribution as well.~\cite{daido1992order, Perez_97, pikovsky2015dynamics, terada2020linear} These higher-order moments are sometimes called Daido order parameters. Daido~\cite{daido1992order} introduced them in his analysis of a generalization of the Kuramoto model, in which the usual coupling function $\sin(\theta_k-\theta_j)$ was replaced by a general periodic function $f(\theta_k-\theta_j)$ containing all possible Fourier harmonics. In Daido's work, and in much of the subsequent work where the higher-order moments $\rho_m$ came into play, the oscillators were assumed to be coupled all-to-all, corresponding to a complete graph. In what follows, we see that $\rho_m$ can also be helpful when analyzing networks of arbitrary topology.  

For an equilibrium $\theta = (\theta_1,\ldots,\theta_n)$, we define the  moments 
\[
\rho_m = \frac{1}{n}\sum_{j} e^{im\theta_j}, \quad m=1,2,\ldots.
\]
(From now on, we use the notation $\sum_{j}$ to mean $\sum_{j=1}^n$.)
Without loss of generality, we may assume that the complex order parameter $\rho_1$ is real-valued and non-negative. To see this, write $\rho_1 = |\rho_1|e^{i\psi}$ for some $\psi$. Then, $\hat{\theta} = (\theta_1-\psi,\ldots,\theta_n-\psi)$ is also an equilibrium of~\eqref{eq:dynamical} with the same stability properties as $\theta$ since~\eqref{eq:dynamical} is invariant under a global shift of all phases by $\psi$.  So for the rest of this paper, we assume that $\psi=0$ for any equilibrium state being considered, with $0 \leq \rho_1 \leq 1$. When $\rho_1=1$, the oscillators are all in phase, corresponding to perfect synchrony, whereas when $\rho_1\approx 0$ the pattern of phases is incoherent.

The higher-order moments $\rho_2,\rho_3,\ldots,$ reveal additional  information about an equilibrium state. When $|\rho_2|=1$, the oscillators form at most two groups: those in sync, with $\cos(\theta_j - \theta_1)=1$, and those in anti-sync, so that $\cos(\theta_j - \theta_1)=-1$. Similarly, $|\rho_m|=1$ reveals that the oscillators form at most $m$ groups. The only equilibrium for which $|\rho_m| = 1$ for all $m \geq 1$ is the all-in-phase state.  

\subsection{A lower bound on the order parameter}
We are particularly interested in the size of $\rho_1$ and $|\rho_2|$ for stable equilibrium states. For an equilibrium to be stable, its first two moments must satisfy an inequality that we now derive. The resulting inequality (and more convenient versions of it that we obtain later) play a crucial role in our proof. 

We begin by deriving a lower bound on $\rho_1$. First, note that $|\rho_m|^2$ satisfies
\[
|\rho_m|^2 = \frac{1}{n^2}\left(\sum_k e^{imk}\sum_j e^{-imj}\right) = \frac{1}{n^2}\sum_{j,k} \cos(m(\theta_k-\theta_j)). 
\]
Since $\cos^2(x-y) = \frac{1}{2}(\cos(2(x-y)) + 1)$, we have 
\begin{equation} 
\begin{aligned}
\frac{1}{n^2}\sum_{j,k} \cos^2(m(\theta_k-\theta_j)) &= \frac{1}{2n^2}\sum_{j,k} \left(\cos(2m(\theta_k-\theta_j)) + 1 \right)\\
&= \frac{1}{2} \left(1 + | \rho_{2m}|^2\right).
\end{aligned}
\label{eq:Fourier}
\end{equation} 
Ling, Xu, and Bandeira~\cite{ling2018landscape} proved (see p.~1893 of their paper) that when $\theta$ is a stable equilibrium, then  
\[
-\sum_{j,k}A_{jk}\cos(\theta_k-\theta_j) + \sum_{j,k}A_{jk}\cos^2(\theta_k-\theta_j)\leq 0
\]
and hence, by~\eqref{eq:Fourier}, we have
\begin{equation} 
\begin{aligned} 
&\sum_{j,k} (1-A_{jk})\left(\cos(\theta_k-\theta_j)-\cos^2(\theta_k-\theta_j)\right) \\
&\leq \sum_{j,k} \!\left(\cos(\theta_k-\theta_j)-\cos^2(\theta_k-\theta_j)\right) = 
\frac{n^2}{2}\!\left( 2\rho_1^2 - |\rho_2|^2 - 1 \right)\!.
\end{aligned} 
\label{eq:NeedAtEnd}
\end{equation} 
In other words, for an equilibrium point to be stable, the following lower bound on $\rho_1^2$ must hold:
\begin{align}
\rho_1^2 \geq & \frac{1+ |\rho_{2}|^2 }{2} \label{eq:lowerbound_r1} \\
& \,\,\,\,\,\,\,\,  +\frac{1}{n^2} \sum_{j,k} (1-A_{jk})\!\left(\cos(\theta_k-\theta_j)-\cos^2(\theta_k-\theta_j)\right).\nonumber
\end{align} 
But this lower bound on $\rho_1^2$ is not convenient for our purposes because it also involves the network's adjacency matrix and the stable equilibrium $\theta$. We adapt this lower bound to a more convenient one in Section~\ref{sec:FinalArgument}. To get there, we need to find a way to replace the dependence on the adjacency matrix and $\theta$ with more aggregated quantities, like $\rho_1, \rho_2$ and $\bmu$. That is the goal of the next two sections. 

\section{A large order parameter implies that a stable equilibrium is the all-in-phase state}
Intuitively, for dense networks, one expects that it is impossible to have a large order parameter $\rho_1$ associated with a stable equilibrium unless it is the all-in-phase state. The idea is that if $\rho_1$ is large, then many of the $\theta_j$'s must be close to $0$, and since the network is also dense, many of these nearly-in-phase oscillators must be connected to each other, which should make it easy for the whole network to fall into sync. One of the key steps to make this intuition precise is the following inequality:
\begin{lemma}
\label{lm:cosbound}
If $\theta$ is a stable equilibrium, then for any $1\leq j\leq n$, we have 
\begin{equation}
n\sqrt{(1-\bmu)^2 - \rho_1^2\sin^2(\theta_j)}\geq \sum_{k} (1-A_{jk})\left|\cos(\theta_k-\theta_j)\right|\geq 0.
\label{eq:cosbound}
\end{equation}
(This inequality holds when there are self-loops as $A_{jj}=1$. For networks without self-loops, the summation $\sum_{k}$ should be replaced with $\sum_{k=1,k\neq j}^{n}$.)
\end{lemma}
\begin{proof}
Select any $j$ such that $1\leq j\leq n$. From the fact that $\theta$ is an equilibrium, we have $\sum_{k}A_{jk}\sin(\theta_k-\theta_j)=0$ and hence
\[
\sum_{k} (1-A_{jk})\sin(\theta_k-\theta_j) = - n \rho_1 \sin (\theta_j).
\] 
Since $1-A_{jk}\geq 0$ and $\sum_{k} (1-A_{jk}) \leq (n-1) (1 -\mu)= n (1 -\bmu)$, the Cauchy--Schwartz inequality shows that
\begin{equation} 
\Big( \sum_{k} (1-A_{jk})\sin(\theta_k-\theta_j) \Big)^2
\!\!\!\leq
n (1-\bmu)\!\sum_{k} \!(1-A_{jk})\sin^2(\theta_k-\theta_j).
\label{eq:SinInequality} 
\end{equation} 
By replacing $\sin^2(\theta_k-\theta_j)$ by $1-\cos^2(\theta_k-\theta_j)$ and using that 
$ \sum_{k} (1-A_{jk})  \leq n (1 -\bmu)$, we find that
\[
n^2 \rho_1^2\sin^2(\theta_j)\! \leq n^2 (1 -\bmu)^2 - n (1 -\bmu)\!\! \sum_{k}\! (1-A_{jk})\cos^2(\theta_k-\theta_j).
\]
The first inequality in the lemma follows from a similar inequality to~\eqref{eq:SinInequality}, but with $\sin$ replaced by $|\cos |$. The second inequality is trivial as each term in $\sum_{k} (1-A_{jk})\left|\cos(\theta_k-\theta_j)\right|$ is non-negative.
\end{proof}
Note that Lemma~\ref{lm:cosbound} implies that for all $1\leq j\leq n$, 
\begin{equation}
\rho_1|\sin(\theta_j) |\leq 1-\bmu
\label{eq:Inequality}
\end{equation} 
since by~\eqref{eq:cosbound} we have $(1-\bmu)^2 - \rho_1^2\sin^2(\theta_j)\geq 0$ . (An equivalent inequality also appears on p.~1895 in Ling, Xu, and Bandeira's paper.~\cite{ling2018landscape})
This allows us to conclude that if $\theta$ is a stable equilibrium associated with a large $\rho_1$, then $\theta$ is the all-in-phase state.
\begin{cor}
\label{cor:LBXstability}
Suppose that $\theta$ is a stable equilibrium such that 
$\rho_1 > \sqrt{2} (1-\bmu)$. Then $\theta$ must be the all-in-phase state.
\end{cor}
\begin{proof}
By~\eqref{eq:Inequality}, we see that 
$|\sin(\theta_j)|\leq (1-\bmu)/\rho_1$ 
for all $j$. 
Since $\rho_1> \sqrt{2}(1-\bmu)$, 
we also have that $|\sin(\theta_j)|<1/\sqrt{2}$ for all $j$. Therefore, $\theta$ must be the all-in-phase state, from Proposition 5 of the paper by Ling, Xu, and Bandeira.~\cite{ling2018landscape} 
\end{proof}
Corollary~\ref{cor:LBXstability} makes our intuition precise and shows us that the only way $\rho_1$ can be large for a stable equilibrium is if it is the all-in-phase state.  

\section{For dense networks the all-in-phase state is the only stable equilibrium}\label{sec:FinalArgument}
We now set out to show that if $\bmu>3/4$, then $\rho_1>\sqrt{2} (1-\bmu)$ for all stable equilibria, which by Corollary~\ref{cor:LBXstability} guarantees that the dense network is globally synchronizing. 

Since $|1 - \cos(\theta) |  \leq 2 $, we find that 
$|\cos(x) - \cos^2(x)| = |\cos(x)(1-\cos(x)) |  \leq 2|\cos(x)|$ and hence, by Lemma~\ref{lm:cosbound}, 
\[
\begin{aligned}
\sum_k &(1-A_{jk}) (\cos(\theta_k-\theta_j)-\cos^2(\theta_k-\theta_j)) \\
&\geq -2n\sqrt{(1-\bmu)^2 - \rho_1^2\sin^2(\theta_j)}
\end{aligned}
\]
for all $1\leq j\leq n$.  Putting this together with~\eqref{eq:lowerbound_r1} leads to the following lower bound on $\rho_1$: 
\begin{equation}
\rho_1^2 \geq \frac{1+|\rho_2|^2}{2} -\frac{2}{n}\sum_{j}\sqrt{(1-\bmu)^2 - \rho_1^2\sin^2(\theta_j)}.
\label{eq:r1bound}    
\end{equation}
This is a more convenient lower bound on $\rho_1^2$ than~\eqref{eq:lowerbound_r1} because~\eqref{eq:r1bound} does not depend on the topology of the network through the adjacency matrix. However, we need to go further and remove the dependency on the $\theta_j$'s. 

If we use the fact that $\rho_1^2\sin^2(\theta_j)\geq 0$, then we find
\begin{equation}
\rho_1^2 \geq 2\left(\bmu - \frac{3}{4}\right) + \frac{1}{2}|\rho_{2}|^2.
\label{eq:BLXbound}    
\end{equation}
The inequality in~\eqref{eq:BLXbound} also implies the one found by Ling, Xu, and Bandeira~\cite{ling2018landscape} (see [5.4] in their paper) by taking the trivial lower bound of $|\rho_2|^2\geq 0$. (To see this, replace $\bmu$ by $(\mu(n-1)+1)/n$ and $\rho_1$ by $\frac{1}{n}\sum_{j}e^{i\theta_j}$.)

To improve the bound obtained by Ling, Xu, and Bandeira,~\cite{ling2018landscape} we now seek a non-trivial lower bound on $|\rho_2|$.
\begin{lemma}
\label{lm:r2bound}
Let $\theta$ be a stable equilibrium with $\rho_1>0$. Then, for all 
$0\leq x_0 \leq \min\{1,(1-\bmu)^2/\rho_1^2\}$, 
the following inequality holds: 
\begin{equation}
|\rho_2| \geq a + b \left( 1+ |\rho_{2}|^2 - 2 \rho_1^2 \right),
\label{eq:r2bound}
\end{equation}
with
\begin{equation}
a = 1+2x_0 - 4\frac{(1-\bmu)^2}{\rho_1^2},
\quad
b = \frac{\sqrt{(1-\bmu)^2-\rho_1^2x_0}}{\rho_1^2}.
\label{eq:abfromx0}
\end{equation}
\end{lemma} 
\begin{proof} 
Since $\cos(2\theta_j) = 1-2\sin^2(\theta_j)$, we have
\[
n|\rho_2|\geq \left|{\rm Re}(n\rho_2)\right|  = \left|\sum_j \cos(2\theta_j)\right| \geq \sum_j (1-2\sin^2(\theta_j)). 
\]
Looking at~\eqref{eq:r1bound} for inspiration to derive a useful expression, we seek a bound of the form:  
\begin{equation}
1-2\sin^2(\theta_j) \geq  a  + 4b\sqrt{(1-\bmu)^2 - \rho_1^2\sin^2(\theta_j)}.
\label{eq:sininquality}
\end{equation}
Since we would like to have~\eqref{eq:sininquality} hold for any possible $\theta_j$, if we write $x = \sin^2 (\theta_j)$ for brevity, then we need
\begin{equation}
1-2x \geq a  + 4 b\sqrt{(1-\bmu)^2 - \rho_1^2x}
\label{eq:ABbound}
\end{equation}
to hold for all 
$0 \leq x \leq \min\{1,(1-\bmu)^2/\rho_1^2\}$. (Note that we can restrict $x$ by~\eqref{eq:Inequality} and the expression for $x$.)
A simple calculation shows that for the choice of $a$ and $b$ in~\eqref{eq:abfromx0}, the graphs of 
\[
f(x) = 1-2x 
\quad \mbox{and} \quad
g(x) = a  + 4 b\sqrt{(1-\bmu)^2 - \rho_1^2x}
\] 
intersect tangentially at $x=x_0$, i.e., $f(x_0)=g(x_0)$ and $f'(x_0)=g'(x_0)$. Moreover, the concavity of the function 
$x\to ((1-\bmu)^2 - \rho_1^2x)^{1/2}$ guarantees that the inequality in~\eqref{eq:ABbound} holds for all 
required values of $x$. We conclude that 
\[
|\rho_2|\geq a + 4b\sum_{j}\sqrt{ (1-\bmu)^2 - \rho_1^2\sin^2(\theta_j)}.
\]
The statement of the lemma follows from~\eqref{eq:r1bound}.
\end{proof} 

It is essential in Lemma~\ref{lm:r2bound} that we assume that $\rho_1>0$ as when $\rho_1 = 0$ one can also have $|\rho_2|=0$. 
What is surprising to us is that even a very small $\rho_1>0$ can lead to a useful lower bound on $|\rho_2|$. In fact, simplifying and then optimizing (over $x_0$) the inequality in Lemma~\ref{lm:r2bound}, we can show that $|\rho_2|\geq 1/2$ for networks with $\bmu>3/4$.

\begin{lemma}
\label{lm:newoptimize}
Suppose that $\bmu > 3/4$ and $\theta$ is a stable equilibrium. Then  $|\rho_2| \geq 1/2$.
\end{lemma}
\begin{proof}
First, note that if $1-2\rho_1^2<0$ then we have $\rho_1> \sqrt{2}/2> 2\sqrt{2}(1-\bmu)$ as $\bmu>3/4$. Therefore, by Corollary~\ref{cor:LBXstability}, $\theta$ must be the all-in-phase state and $\rho_2=1$. Thus for the remainder of this proof, we assume that $1-2\rho_1^2\geq 0$. 

When $1-2\rho_1^2\geq 0$, we show that the inequalities~\eqref{eq:BLXbound} and~\eqref{eq:r2bound} with $a$ and $b$ given
in~\eqref{eq:abfromx0} cannot both be satisfied for all $0<x_0\leq \min\{1,(1-\bmu)^2/\rho_1^2\}$ unless $|\rho_2|\geq1/2$.
Since $|\rho_2|^2\geq 0$ we have $\rho_1^2\geq 2(\bmu - 3/4)>0$ from~\eqref{eq:BLXbound}, and $|\rho_2|\geq a +b(1-2\rho_1^2)$ from~\eqref{eq:r2bound}.
A simple calculation shows
that the value of $x_0$ that optimizes the lower bound $|\rho_2|\geq a +b(1-2\rho_1^2)$, where $a$ and $b$ are in~\eqref{eq:abfromx0} is given by
\[
x_0^* = \frac{(1-\bmu)^2}{\rho_1^2} - \frac{\left(1 - 2\rho_1^2\right)^2}{16\rho_1^2}.
\]
Clearly, $x_0^*\leq (1-\bmu)^2/\rho_1^2$ and we find that $x_0^*\geq 0$ because $\rho_1^2\geq 2(\bmu - 3/4)$. Moreover, $x_0^*<1$ since $16(1-\bmu)^2 - (1-2\rho_1^2)^2< 16\rho_1$ for $3/4<\bmu\leq 1$. For this valid choice of $x_0^*$, we find that
\[
|\rho_2| \geq 1 - \frac{2(1-\bmu)^2}{\rho_1^2} + \frac{(1-2\rho_1^2)^2}{8\rho_1^2} = 1 - 2x_0^*.
\]
The statement of the lemma follows by noting that $1-2x_0^*\geq 1/2$ if and only if $\rho_1^4 \geq (16(1-\bmu)^2 - 1)/4$. This last inequality holds as $(16(1-\bmu)^2 - 1)/4$ is negative when $\bmu>3/4$. 
\end{proof}

Finally, we are ready to prove our main result. 
\begin{thm}
If \thinspace $\bmu > 3/4$, then the only stable equilibrium is the all-in-phase state.
\label{thm:MainResult}
\end{thm}
\begin{proof}
By Lemma~\ref{lm:newoptimize}, we know that $|\rho_2| \geq 1/2$. By~\eqref{eq:BLXbound}, we find that $\rho_1^2\geq \frac{1}{2}|\rho_2|^2$ for $\bmu>3/4$. Thus, $\rho_1^2\geq 1/8$. 
To conclude the proof, we just need to ensure that this implies that $\rho_1> \sqrt{2}(1-\bmu)$ (see Corollary~\ref{cor:LBXstability}). One can easily see that 
$1/8> 2 (1-\bmu)^2$ when $\bmu>3/4$.
\end{proof}
Theorem~\ref{thm:MainResult} says that any network with self-loops and connectivity $\bmu>0.75$ must be globally synchronizing. As $\bmu = (\mu(n-1)+1)/n$, we know that any network of size $n$ without self-loops and $\mu>(0.75n-1)/(n-1)$ is also guaranteed to be globally synchronizing (see Section~\ref{sec:twinning}). Finally, since $(0.75n-1)/(n-1)\rightarrow 0.75$ as $n\rightarrow \infty$ from below, any network without self-loops and connectivity $\mu \geq 0.75$ cannot support a pattern. We conclude that $\mu_c\leq 0.75$. 

\section{Networks with connectivity just below three-quarters}\label{sec:JustBelow}
This section discusses what our argument can say about the stable equilibria of dense networks whose connectivity is just below three-quarters. We will roughly sketch an argument here without attempting to make it precise.

At first glance, it might seem that we cannot say anything because when $\bmu \leq 0.75$ the inequalities in~\eqref{eq:BLXbound} and~\eqref{eq:r2bound} can both be satisfied with $\rho_1 = 0$ and $|\rho_2|=0$. However, we can describe the possible stable equilibria when $\bmu$ is just below $0.75$. For example, when $\bmu \geq 0.7495$, 
one can see that if both the inequalities in~\eqref{eq:BLXbound} and~\eqref{eq:r2bound} are satisfied then either: (i) $\rho_1 > 0.7065$ 
or (ii) $\rho_1 < 0.03166$ and $|\rho_2| < 0.04474$. (These bounds are computed by searching over $(\rho_1,\rho_2)\in [0,1]^2$, optimizing for $x_0$ in Lemma~\ref{lm:r2bound}, and checking to see when both~\eqref{eq:BLXbound} and~\eqref{eq:r2bound} are satisfied.) 
   
In case (i), we see that $\rho_1>\sqrt{2}(1-\bmu) \geq 0.35$ so by Corollary~\ref{cor:LBXstability} this situation corresponds to the all-in-phase state. Case (ii) is more interesting as it represents the possibility of a stable pattern. From~\eqref{eq:NeedAtEnd} and our bounds on $\rho_1$ and $|\rho_2|$, we find that  
\begin{equation} 
\begin{aligned} 
\sum_{j,k} (1-A_{jk}) & \left(\cos(\theta_k-\theta_j)-\cos^2(\theta_k-\theta_j)\right) \\
& \leq -0.49900 n^2 \leq -1.9921(1 - \bmu)n^2.
\end{aligned} 
\label{eq:NearlyLower}
\end{equation} 
Since $\cos(\theta_k-\theta_j)-\cos^2(\theta_k-\theta_j)\geq -2$ and $\sum_{j,k} (1-A_{jk}) = (1-\bmu)n^2$, the average contribution in~\eqref{eq:NearlyLower} from two oscillators $j$ and $k$ that are not connected is at most $-1.9921$. This means that in the vast majority of cases where one has $A_{jk} = 0$, we also have $\cos(\theta_j-\theta_k)\approx -1$, so oscillators $j$ and $k$ must be  nearly anti-synchronized with a phase difference of roughly $\pi$ between them. 

Therefore, there are at least two clusters of oscillators in the network of size $\geq 0.249n$ that are nearly in anti-sync. Inside the clusters, the phases differ by at most $0.146$ radians (about $8.4$ degrees). The bound for the size of the size of the two clusters comes from $1.9921(1-\bmu)/2$, and the phase spread comes from doubling the smallest positive root of $\cos(\pi- \phi) - \cos^2(\pi-\phi) = -1.9921$. The bounds on $\rho_1$ and $|\rho_2 |$ in~\eqref{eq:BLXbound} and~\eqref{eq:r2bound} further imply that there are two more clusters of size $\geq 0.249n$ with phases that are shifted by approximately $\pi/2$ compared to the other pair of clusters. This is because the first identified pair of clusters makes a significant contribution to $|\rho_2|$ but $|\rho_2|$ is tiny.  Therefore, the only way for $|\rho_2|$ to remain small is if there is another pair of clusters that approximately cancel out the contribution of the first pair. Thus, for networks with $\bmu > 0.7495$, the vast majority of the oscillators fall into four clusters with phases that are at $\phi$, $\phi+\pi/2$, $\phi+\pi$, and $\phi+3\pi/2$ for some $\phi$. In addition, there are at most $n/250$ rogue oscillators in the network that do not fit into those four clusters.

The upshot of this argument is that any network with connectivity $0.7495<\bmu\leq 0.75$ can either be globally synchronizing or can support a particular pattern of the type we have just described. All our attempts to construct a network in this regime with such a pattern have not been successful so far. We believe that such networks do not support \emph{linearly stable} patterns of this type (or any other type), but we are currently unable to rule them out rigorously. If no such networks or patterns exist, or if the patterns do exist but are only weakly (nonlinearly) stable, then this would be surprising. Indeed, it would suggest something remarkable: that the gap between the lower and upper bound $0.6838< \mu_c \leq 0.75$ might not be bridgeable by linear stability analysis alone, because of a minefield of patterns on their own razor's edge of stability. 

\appendix




\begin{acknowledgments}
We thank Lee DeVille for pointing out the connection between self-loops and twinning in the summer of 2020, which we used here to simplify the presentation of our argument. We also thank Federico Fuentes for comments on a draft. Research supported by Simons Foundation Grant 713557 to M.~K., NSF Grants No.~DMS-1513179 and No.~CCF-1522054 to S.~H.~S., and NSF Grants No.~DMS-1818757, No.~DMS-1952757, and No.~DMS-2045646 to A.~T. 
\end{acknowledgments}

\section*{Data Availability}
The data that supports the findings of this study are almost all available within the article. Any data that is not available can be found from the corresponding author upon reasonable request.


\begin{thebibliography}{10}


\bibitem{winfree1967biological}
A.~T.~Winfree.
\newblock Biological rhythms and the behavior of populations of coupled oscillators.
\newblock {\em J. Theor. Bio.} 16, 15 (1967).

\bibitem{peskin1975mathematical}
C.~S.~Peskin.
\newblock {\em Mathematical Aspects of Heart Physiology}
\newblock {Courant Inst. Math. Sci.}, 268 (1975).

\bibitem{kuramoto1984chemical}
Y.~Kuramoto.
\newblock {\em Chemical Oscillations, Waves, and Turbulence}.
\newblock Springer, 1984.

\bibitem{mirollo1990synchronization}
R.~E.~Mirollo and S.~H.~Strogatz.
\newblock Synchronization of pulse-coupled biological oscillators.
\newblock {\em SIAM J. Appl. Math.} 50, 1645 (1990).

\bibitem{watanabe1994constants}
S.~Watanabe and S.~H.~Strogatz.
\newblock Constants of motion for superconducting Josephson arrays.
\newblock {\em Physica D: Nonlinear Phenomena} 74, 197 (1994).

\bibitem{pikovsky2003synchronization}
A.~Pikovsky, M.~Rosenblum, and J.~Kurths.
\newblock {\em Synchronization: A Universal Concept in Nonlinear Sciences}, vol.~12,
\newblock Cambridge University Press, 2003.

\bibitem{matheny2019exotic}
M.~H.~Matheny, J.~Emenheiser, W.~Fon, A.~Chapman, A.~Salova, M.~Rohden, J.~Li, M.~H.~de~Badyn, M.~P{\'o}sfai, L.~Duenas-Osorio, et~al.
\newblock Exotic states in a simple network of nanoelectromechanical oscillators.
\newblock {\em Science} 363, eaav7932 (2019).

\bibitem{jadbabaie2004stability}
A.~Jadbabaie, N.~Motee, and M.~Barahona.
\newblock On the stability of the Kuramoto model of coupled nonlinear
  oscillators.
\newblock In {\em Proc. 2004 Amer. Contr. Conf.} 5, 4296 (2004).

\bibitem{wiley2006size}
D.~A Wiley, S.~H.~Strogatz, and M.~Girvan.
\newblock The size of the sync basin.
\newblock {\em Chaos} 16, 015103 (2006).

\bibitem{mallada2010synchronization}
E.~Mallada and A.~Tang.
\newblock Synchronization of phase-coupled oscillators with arbitrary topology.
\newblock In {\em Proc. 2010 Amer. Contr. Conf.}, 1777 (2010).

\bibitem{taylor2012there}
R.~Taylor.
\newblock There is no non-zero stable fixed point for dense networks in the homogeneous Kuramoto model.
\newblock {\em J. Phys. A: Math. Theor.} 45, 055102 (2012).

\bibitem{dorfler2014synchronization}
F.~D{\"o}rfler and F.~Bullo.
\newblock Synchronization in complex networks of phase oscillators: A survey.
\newblock {\em Automatica} 50, 1539 (2014).

\bibitem{pecora2014cluster}
L.~M.~Pecora, F.~Sorrentino, A.~M.~Hagerstrom, T.~E.~Murphy, and R.~Roy.
\newblock Cluster synchronization and isolated desynchronization in complex networks with symmetries.
\newblock {\em Nature Comm.} 5, 4079 (2014).

\bibitem{pikovsky2015dynamics}
A.~Pikovsky and M.~Rosenblum.
\newblock Dynamics of globally coupled oscillators: Progress and perspectives.
\newblock {\em Chaos} 25, 097616 (2015).

\bibitem{canale2015exotic}
E.~A.~Canale and P.~Monz{\'o}n.
\newblock Exotic equilibria of Harary graphs and a new minimum degree lower bound for synchronization.
\newblock {\em Chaos} 25, 023106 (2015).
  
\bibitem{Mehta15} 
D.~Mehta, N.~S.~Daleo, F.~D\"{o}rfler, and J.~D.~Hauenstein. 
\newblock Algebraic geometrization of the Kuramoto model: Equilibria and stability analysis.
\newblock {\em Chaos} 25, 053103 (2015).

\bibitem{rodrigues2016kuramoto}
F.~A.~Rodrigues, T.~K.~DM.~Peron, P.~Ji, and J.~Kurths.
\newblock The Kuramoto model in complex networks.
\newblock {\em Phys. Reports} 610, 1 (2016).

\bibitem{abrams2016introduction}
D.~M Abrams, L.~M Pecora, and A.~E.~Motter.
\newblock Introduction to focus issue: Patterns of network synchronization.
\newblock {\em Chaos} 26, 094601 (2016).

\bibitem{deville2016phase}
L.~DeVille and B.~Ermentrout.
\newblock Phase-locked patterns of the Kuramoto model on 3-regular graphs.
\newblock {\em Chaos} 26, 094820 (2016).

\bibitem{sokolov2018sync}
Y.~Sokolov and G.~B.~Ermentrout.
\newblock When is sync globally stable in sparse networks of identical Kuramoto oscillators?
\newblock {\em Physica A: Statistical Mechanics and its Applications} 533, 122070 (2019).

\bibitem{ling2018landscape}
S.~Ling, R.~Xu, and A.~S.~Bandeira.
\newblock On the landscape of synchronization networks: A perspective from nonconvex optimization.
\newblock {\em SIAM J. Optim} 29, 1807 (2019).

\bibitem{lu2019synchronization}
J.~Lu and S.~Steinerberger.
\newblock Synchronization of Kuramoto oscillators in dense networks.
\newblock {\em Nonlinearity} 33, 5905 (2020).








  
\bibitem{townsendstillmanstrogatz}
A. Townsend, M. Stillman, and S. H. Strogatz.
\newblock Dense networks that do not synchronize and sparse ones that do.
\newblock {\em Chaos} 30, 083142 (2020).
 
\bibitem{Yoneda_2021}
R. Yoneda, T. Tatsukawa, and J. Teramae. 
\newblock The lower bound of the network connectivity guaranteeing in-phase synchronization.
\newblock {\em arXiv preprint arXiv:2104.05954}, 2021.

\bibitem{kuramoto} 
Y. Kuramoto. 
\newblock Self-entrainment of a population of coupled non-linear oscillators.
\newblock {\em International Symposium on Mathematical Problems in Theoretical Physics. Springer, Berlin}, 1975.

\bibitem{daido1992order}
H. Daido. 
\newblock Order function and macroscopic mutual entrainment in uniformly coupled limit-cycle oscillators. 
\newblock {\em Progress of Theoretical Physics} 88, 1213 (1992).

\bibitem{Perez_97} 
J. C. Perez and F. Ritort.
\newblock A moment-based approach to the dynamical solution of the Kuramoto model.
\newblock {\em J. Phys. A: Math. Gen.} 30, 8095 (1997).

\bibitem{terada2020linear} 
Yu Terada and Yoshiyuki Y. Yamaguchi. 
\newblock Linear response theory for coupled phase oscillators with general coupling functions.
\newblock {\em J. Phys. A: Math. Gen.} 53, 044001 (2020).

\end{thebibliography}

\end{document}